\documentclass[12pt,reqno]{amsart}
\usepackage{mathrsfs}
\usepackage{amsfonts}
\usepackage{amsfonts}
\usepackage{amsfonts}
\usepackage{units}
\usepackage{amssymb,amsmath,amsthm}
\usepackage{abstract}
\usepackage{amsfonts}
\usepackage{amsfonts}
\usepackage{amsfonts}
\usepackage{amsfonts}
\usepackage{amsfonts}
\usepackage{amsfonts}
\usepackage{amsfonts}
\usepackage{amsfonts}
\usepackage{amsfonts}
\usepackage{amsfonts}
\usepackage{amsfonts}
\usepackage{amsfonts}
\usepackage{a4wide}
\newtheorem{lemma}{Lemma}[section]

\newtheorem{theorem}{Theorem}[section]

\def\theequation{\thesection.\arabic{equation}}
\let\Section=\section
\def\section{\setcounter{equation}{0}\Section}
\sf
\begin{document}
\title[A class of  nonlocal elliptic
problems in the half space with a hole]
{Existence of positive solution for a class of nonlocal elliptic
problems in the half space with a hole}
\author{Xing Yi}
\thanks{*College of Mathematics and Computer Science, Key Laboratory of High Performance Computing
and Stochastic Information Processing (Ministry of Education of China),
Hunan Normal University, Changsha, Hunan 410081, P. R. China
(hnyixing522@hotmail.com)}
\maketitle

\vskip 0.3in
{\bf Abstract}\quad This work concerns with the existence of solutions for the following class of nonlocal elliptic problems
\begin{eqnarray}\label{eq:0.1}
&&\left\{\begin{array}{l}
(-\Delta)^{s} u+u=|u|^{p-2} u \text { in } \Omega_{r} \\
u \geq 0 \quad \text { in }\Omega_{r} \text { and } u \neq 0 \\
u=0 \quad \mathbb{R}^{N} \backslash \Omega_{r}
\end{array}\right.,
\end{eqnarray}
involving the fractional Laplacian operator $(-\Delta)^{s},$ where $s \in(0,1), N>2 s$, $\Omega_{r}$ is the half space with a hole in $\mathbb{R}^N$  and $p \in\left(2,2_{s}^{*}\right) .$ The main technical approach is based on variational and topological methods.
\vskip 0.1in
\noindent{\it Keywords:}\quad  Nonlocal elliptic problems,  Positive high energy solution, Half space with a hole

\noindent {\bf AMS} classification:  58J05,  35J60.  \vspace{3mm}

\renewcommand{\theequation}{\thesection.\arabic{equation}}
\section*{1. Introduction}
\setcounter{section}{1}\setcounter{equation}{0}
Let $\mathbb{R}_{+}^{N}=\left\{\left(x^{\prime}, x_{N}\right) \in \mathbb{R}^{N-1} \times \mathbb{R} \mid 0<x_{N}<\infty\right\}$ be the upper half space. $\Omega_{r}$ is  an unbounded smooth domain such that
\[\overline{\Omega_{r}}\subset \mathbb{R}_{+}^{N},\]
 and \[ \mathbb{R}_{+}^{N}\setminus\overline{\Omega_{r}}\subset B_{\rho}(a_{r})\subset \mathbb{R}_{+}^{N}\  \] with $a_{r}=(a,r)\in \mathbb{R}_{+}^{N}$.
 Indeed, $\Omega_r$  is the upper half space with a hole.

 We consider the following fractional elliptic problem:
\begin{eqnarray}\label{eq:1.1}
&&\left\{\begin{array}{l}
(-\Delta)^{s} u+u=|u|^{p-2} u \text { in } \Omega \\
u \geq 0 \quad \text { in }\Omega_{r} \text { and } u \neq 0 \\
u=0 \quad \mathbb{R}^{N} \backslash \Omega
\end{array}\right.,
\end{eqnarray}
where $ \Omega=\Omega_r,\ s \in(0,1), N>2 s$, $p \in\left(2,2_{s}^{*}\right),$ where $2_{s}^{*}=\frac{2 N}{N-2 s}$ is the fractional critical Sobolev exponent and $(-\Delta)^{s}$ is the classical fractional Laplace operator.

When $s \nearrow 1^{-},$ problem(\ref{eq:1.1})  is related to  the following elliptic problem
\begin{eqnarray}\label{eq:1.2}
-\triangle u+u=|u|^{p-1}u,\quad x\in \Omega, \quad u\in H_{0}^{1}(\Omega)
\end{eqnarray}
.
When $\Omega$ is a bounded domain, by applying the compactness of the embedding $H_{0}^{1}(\Omega)\hookrightarrow L^{p}(\Omega), 1< p<\frac{2N}{N-2}$, there is a positive solution of (\ref{eq:1.2}). If $\Omega$ is an unbounded domain, we can not obtain a solution for problem (\ref{eq:1.2}) by using Mountain-Pass Theorem directly
because the embedding $H_{0}^{1}(\Omega)\hookrightarrow L^{p}(\Omega), 1< p<\frac{2N}{N-2}$ is not compactness. However, if $\Omega=\mathbb{R}^{N}$, Berestycki-Lions \cite{21},
 proved that there is a radial positive solution of equation (\ref{eq:1.2})
  by applying the compactness of the embedding $H_{r}^{1}(\mathbb{R}^{N})\hookrightarrow L^{p}(R^{N}),2<p<\frac{2N}{N-2}$, where $H_r^{1}(\mathbb{R}^{N})$ consists of the radially symmetric functions in $H^{1}(\mathbb{R}^{N})$.
  By the P.L.Lions's Concentration-Compactness Principle \cite{7}, there exists an unique positive solution for problem (\ref{eq:1.2}) in $\mathbb{R}^{N}$.
  By moving Plane method,  Gidas-Ni-Nirenberg \cite{3} also proved that every positive solution  of equation
  \begin{eqnarray}\label{eq:1.3}
-\triangle u+u=|u|^{p-1}u,\quad x\in \mathbb{R}^{N}, \quad u\in H^{1}(\mathbb{R}^{N})
\end{eqnarray}
 is radially symmetric with respect to some point in $\mathbb{R}^{N}$ satisfying
\begin{eqnarray}
u(r)re^{r}=\gamma+o(1)\  as  \ r \rightarrow \infty.
\end{eqnarray}
Kwong \cite{188} proved that the positive solution of (\ref{eq:1.3}) is unique up to translations.

  In fact,  Esteban and Lions \cite{1} proved that there
is not any nontrivial solution of equation  (\ref{eq:1.2}) when $\Omega$ is an Esteban-Lions domain (for example $\mathbb{R}_+^3$). Thus, we want to  change the topological property of the domain $\Omega$ to look for a  solution of problem (\ref{eq:1.2}). Wang \cite{4} proved
that if $\rho$ is sufficiently small and $z_{0N}\rightarrow\infty$, then Eq(\ref{eq:1.2}) admits a positive higher energy
solution in $\mathbb{R}_{+}^{N} \backslash \overline{B_{\rho}\left(z_{0}^{\prime}, z_{0 N}\right)}$. Such problem has been extensively studied in recent
years, see for instance, \cite{9,55} and references therein. From the above researches,
 we believed that  the existence of the solution to the equation (\ref{eq:1.2}) will be affected by the topological property of the domain $\Omega$.

Recently, the case $s \in(0,1)$ has received a special attention, because involves the fractional Laplacian operator $(-\Delta)^{s}$, which arises in a quite natural way in many different contexts, such as, among the others, the thin obstacle problem, optimization, finance, phase transitions, stratified materials, anomalous diffusion, crystal dislocation, soft thin films, semipermeable membranes, flame propagation, conservation laws, ultra-relativistic limits of quantum mechanics, quasigeostrophic flows, multiple scattering, minimal surfaces, materials science and water waves, for more detail see \cite{x1,e16,x2,D21,x4}.

When $ \Omega \subset \mathbb{R}^{N}$ is an exterior domain, i.e. an unbounded domain with smooth boundary $\partial \Omega \neq \emptyset$ such that $\mathbb{R}^{N} \backslash \Omega$ is bounded, $s \in(0,1), N>2 s$,  $p \in\left(2,2_{s}^{*}\right),$  the above problem has been studied by O. Alves, Giovanni Molica Bisci , César E. Torres Ledesma in \cite{21} proving that (1.1) does not have a ground state solution. This fact represents a serious difficulty when dealing with this kind of nonlinear fractional elliptic phenomena. More precisely, the authors analyzed the behavior of Palais-Smale sequences and showed a precise estimate of the energy levels where the Palais-Smale condition fails, which made it possible to show that the problem (1.1) has at least one positive solution, for $\mathbb{R}^{N} \backslash \Omega$ small enough.  A key point in the approach explored in \cite{16,9} is the existence and uniqueness, up to a translation, of a positive solution $Q$  of the limit problem associated with (\ref{eq:1.1}) given by
\begin{eqnarray}\label{eq:1.6}
(-\Delta)^{s} u+u=|u|^{p-2} u \text { in } \mathbb{R}^{N},
\end{eqnarray}
for every $p \in\left(2,2_{s}^{*}\right)$. Moreover, $Q$ is radially symmetric about the origin and monotonically decreasing in $|x| .$ On the contrary of the classical elliptic case,  the exponential decay at infinity is not used in order to prove the existence of a nonnegative solution for (\ref{eq:1.1}).

When $\Omega \subset \mathbb{R}^{N}$ is $\mathbb{R}_+^N$, by moving Plane method,  Wenxiong Chen, Yan Li and Pei Ma \cite{11d6}(p123 Theorem 6.8.3 ) proved that there is no nontrivial solution of problem (\ref{eq:1.1}).
   It is interesting in considering the existence of the high energy equation for the problem (\ref{eq:1.1})
in the half space with a hole in $\mathbb{R}^{N}$.
\begin{theorem}
There is $\rho_{0}>0,r_{0}>0$ such that if $0<\rho\leq\rho_{0} $ and $r \geq r_{0}$,
then there is a positive solution of equation (\ref{eq:0.1}).
\end{theorem}

The paper is organized as follows. In section 2, we give some preliminary results. The Compactness lemma will be given in Section 3. At last, we give the proof of  Theorem 1.1.

\section{Some preliminary results }
For $s \in(0,1)$ and $N>2 s,$ the fractional Sobolev space of order $s$ on $\mathbb{R}^{N}$ is defined by
$$
H^{s}\left(\mathbb{R}^{N}\right):=\left\{u \in L^{2}\left(\mathbb{R}^{N}\right): \int_{\mathbb{R}^{N}} \int_{\mathbb{R}^{N}} \frac{|u(x)-u(z)|^{2}}{|x-z|^{N+2 s}} d z d x<\infty\right\}
$$
endowed with the norm
$$
\|u\|_{s}:=\left(\int_{\mathbb{R}^{N}}|u(x)|^{2} d x+\int_{\mathbb{R}^{N}} \int_{\mathbb{R}^{N}} \frac{|u(x)-u(z)|^{2}}{|x-z|^{N+2 s}} d z d x\right)^{1 / 2}
.$$
We recall the fractional version of the Sobolev embeddings (see \cite{q16}).

\begin{theorem}
Let  $s \in(0,1),$ then there exists a positive constant $C=C(N, s)>0$ such that
$$
\|u\|_{L^{2_{s}^{*}}\left(\mathbb{R}^{N}\right)}^{2} \leq C \int_{\mathbb{R}^{N}} \int_{\mathbb{R}^{N}} \frac{|u(x)-u(y)|^{2}}{|x-y|^{N+2 s}} d y d x
$$
and then $H^{s}\left(\mathbb{R}^{N}\right) \hookrightarrow L^{q}\left(\mathbb{R}^{N}\right)$ is continuous for all $q \in\left[2,2_{s}^{*}\right] .$ Moreover, if $\Theta \subset \mathbb{R}^{N}$ is a
bounded domain, we have that the embedding $H^{s}\left(\mathbb{R}^{N}\right) \hookrightarrow L^{q}(\Theta)$ is compact for any $q \in\left[2,2_{s}^{*}\right) .$
\end{theorem}
Hereafter, we denote by $X_{0}^{s} \subset H^{s}\left(\mathbb{R}^{N}\right)$ the subspace defined by
$$
X_{0}^{s}:=\left\{u \in H^{s}\left(\mathbb{R}^{N}\right): u=0 \text { a.e. in } \mathbb{R}^{N} \backslash \Omega\right\}.
$$
We endow $X_{0}^{s}$ with the norm $\|\cdot\|_{s}$. Moreover we introduce the following norm
$$
\|u\|:=\left(\int_{\Omega_{r}}|u(x)|^{2} d x+\iint_{\mathcal{Q}} \frac{|u(x)-u(z)|^{2}}{|x-z|^{N+2 s}} d z d x\right)^{\frac{1}{2}}
$$
where $\mathcal{Q}:=\mathbb{R}^{2 N} \backslash\left(\Omega_{r}^{c} \times \Omega_{r}^{c}\right)$. We point out that $\|u\|_{s}=\|u\|$ for any $u \in X_{0}^{s}$. Since $\partial \Omega$ is bounded and smooth, by [\cite{D21}, Theorem 2.6], we have the following result.
\begin{theorem}
The space $C_{0}^{\infty}(\Omega)$ is dense in $\left(X_{0}^{s},\|\cdot\|\right) .$
\end{theorem}
In what follows, we denote by $H^{s}(\Omega)$ the usual fractional Sobolev space endowed with the norm
$$
\|u\|_{H^{s}}:=\left(\int_{\Omega}|u(x)|^{2} d x+\int_{\Omega} \int_{\Omega} \frac{|u(x)-u(z)|^{2}}{|x-z|^{N+2 s}} d z d x\right)^{\frac{1}{2}}.
$$
Related to these fractional spaces, we have the following properties

{\bf Proposition 2.3.}The following assertions hold true:

(i) If $v \in X_{0}^{s}$, we have that $v \in H^{s}(\Omega)$ and
$$
\|v\|_{H^{s}} \leq\|v\|_{s}=\|v\| .
$$

(ii) Let $\Theta$ an open set with continuous boundary. Then, there exists a positive constant $\mathfrak{C}=$ $\mathfrak{C}(N, s),$ such that
$$
\|v\|_{L^{2_{s}^{*}}(\Theta)}^{2}=\|v\|_{L^{2^{*}}\left(\mathbb{R}^{N}\right)}^{2} \leq \mathfrak{C} \iint_{\mathbb{R}^{2 N}} \frac{|v(x)-v(z)|^{2}}{|x-z|^{N+2 s}} d z d x
$$
for every $v \in X_{0}^{s} ;$ see [\cite{e16}, Theorem 6.5$]$.

From now on, $M_{\infty}$ denotes the following constant
\begin{eqnarray}\label{eq:1.s6}
M_{\infty}:=\inf \left\{\|u\|_{s}^{2}: u \in H^{s}\left(\mathbb{R}^{N}\right) \text { and } \int_{\mathbb{R}^{N}}|u(x)|^{p} d x=1\right\},
\end{eqnarray}
which is positive by Theorem 2.1. Furthermore, for any $v \in H^{s}\left(\mathbb{R}^{N}\right)$ and $z \in \mathbb{R}^{N},$ we set the function
$$
v^{z}(x):=v(x+z).
$$
Then, by doing the change of variable $\tilde{x}=x+z$ and $\tilde{y}=y+z,$ it is easily seen that
$$\left\|v^{z}\right\|_{s}^{2}=\|v\|_{s}^{2} \text { as well as }\left\|v^{z}\right\|_{L^{p}\left(\mathbb{R}^{N}\right)}=\|v\|_{L^{p}\left(\mathbb{R}^{N}\right)} .$$
Arguing as in \cite{b16} the following result holds true.

\begin{theorem}
Let $\left\{u_{n}\right\} \subset H^{s}\left(\mathbb{R}^{N}\right)$ be a minimizing sequence such that
$$
\left\|u_{n}\right\|_{L^{p}\left(\mathbb{R}^{N}\right)}=1 \text { and }\left\|u_{n}\right\|_{s}^{2} \rightarrow M_{\infty} \text { as } n \rightarrow+\infty.
$$
Then, there is a sequence $\left\{y_{n}\right\} \subset \mathbb{R}^{N}$ such that $\left\{u_{n}^{y_{n}}\right\}$ has a convergent subsequence, and so, $M_{\infty}$ is attained.
\end{theorem}
As a byproduct of the above result the next corollary is obtained.

 {\bf Corollary 1}  There is $v \in H^{s}\left(\mathbb{R}^{N}\right)$ such that $\|v\|_{s}=M_{\infty}$ and $\|v\|_{L^{p}\left(\mathbb{R}^{N}\right)}=1 .$

Let $\varphi$ be a minimizer of $(\ref{eq:1.s6}),$ that is
$$
\varphi \in H^{s}\left(\mathbb{R}^{N}\right), \quad \int|\varphi|^{p} d x=1 \text { and } M_{\infty}=\|\varphi\|_{s}^{2}.
$$
Take \begin{eqnarray}
&&\xi \in C^{\infty}(\mathbb{R}^{+},\mathbb{R}),\eta\in  C^{\infty}(\mathbb{R},\mathbb{R}),
\end{eqnarray}
such that
\[\xi(t)=\left\{
 \begin{array}{ll}
0,0\leq t\leq \rho,\\[2mm]
1,t\geq 2\rho,
\end{array}
\right.\]
\[ \eta(t)=\left\{
 \begin{array}{ll}
0,t\leq 0,\\[2mm]
1,t\geq 1,
\end{array}
\right.\]
and \[0\leq \zeta\leq 1,0\leq \eta\leq 1.
\]

Now, we define
\[f_{y}(x)=\xi(|x-a_{r}|)\eta(x_{N})\varphi(x-y),\]
and
\[\Psi_{y}(x)=\frac{f_{y}(x)}{\left\|f_{y}\right\|_{L^{p}\left(\mathbb{R}^{N}\right)}}=c_{y} f_{y}(x) \text { where } c_{y}=\frac{1}{\left\|f_{y}\right\|_{L^{p}\left(\mathbb{R}^{N}\right)}} .\]
Throughout this section we endow $X_{0}^{s}$ with the norm
$$
\|u\|:=\left(\iint_{\mathcal{Q}} \frac{|u(x)-u(y)|^{2}}{|x-y|^{n+2 s}} d y d x+\int_{\Omega_{r}}|u|^{2} d x\right)^{1 / 2}
$$
and denote by $M>0$ the number
\begin{eqnarray}\label{eq:1w.sq6}
M:=\inf \left\{\|u\|^{2}: u \in X_{0}^{s}, \int_{\Omega_{r}}|u(x)|^{p} d x=1\right\} .
\end{eqnarray}
\begin{lemma}\label{lm:2.4}
Let $y=(y^{'},y_{N})$ , we have\\
(1)$\|f_{y}-\varphi(x-y)\|_{L^{p}(\mathbb{R}^{N})}=o(1)$, $|y-a_{r}|\rightarrow \infty$, and $y_{N}\rightarrow +\infty$, or $y_{N} \rightarrow \infty$ and $\rho \rightarrow 0$;
(2)$\|f_{y}-\varphi(x-y)\|=o(1),|y-a_{r}|\rightarrow \infty$, and $y_{N}\rightarrow +\infty$, or $y_{N} \rightarrow +\infty$ and $\rho \rightarrow 0$.\label{11}
\end{lemma}
\begin{proof}
Similarly as \cite{9,4}, we have

(i) After the change of variables $z=x-y,$  one has
\[
\begin{array}{ll}
\ \|f_{y}-\varphi(x-y)\|^{p}_{L^{p}(\mathbb{R}^{N})}
&=\int_{\mathbb{R}^{N}}|\xi(|x-a_{r}|)\eta(x_{N})-1|^{p}|\varphi(x-y)|^{p}\ dx\\[2mm]
&=\int_{\mathbb{R}^{N}}|\xi(|x+y-a_{r}|)\eta(x_{N}+y)-1|^{p}|\varphi(z)|^{p}\ dz
\end{array}
\]
Let $g_{y}(z)=|\xi(|x+y-a_{r}|)\eta(x_{N}+y)-1|^{p}|\varphi(z)|^{p}.$ Since $|y-a_{r}|\rightarrow \infty$, and $y_{N}\rightarrow +\infty$, it follows that
$$
g_{y}(z) \rightarrow 0 \text { a.e. in } \mathbb{R}^{N}
$$
Now, taking into account that
$$
g_{y}(z)=|\xi(|x+y-a_{r}|)\eta(x_{N}+y)-1|^{p}|\varphi(z)|^{p} \leq 2^{p}|\varphi(z)|^{p} \in L^{1}\left(\mathbb{R}^{N}\right),
$$
the Lebesgue's dominated convergence theorem yields
$$
\int_{\mathbb{R}^{N}} g_{y}(z) d z \rightarrow 0 \text { as } |y-a_{r}|\rightarrow \infty, \  y_{N}\rightarrow +\infty.
$$
Therefore
$$
\|f_{y}-\varphi(x-y)\|_{L^{p}(\mathbb{R}^{N})}=o(1), |y-a_{r}|\rightarrow \infty, \  y_{N}\rightarrow +\infty
$$
\[
\begin{array}{ll}
\ \|f_{y}-\varphi(x-y)\|^{p}_{L^{p}(\mathbb{R}^{N})}
&=\int_{B_{2\rho(a_{r})}\cup\{x_{N}\leq 1\}}|\xi(|x-a_{r}|)\eta(x_{N})-1|^{p}|\varphi(x-y)|^{p}\ dx\\[2mm]
&=\int_{B_{2\rho(a_{r})}}|\xi(|x-a_{r}|)\eta(x_{N})-1|^{p}|\varphi(x-y)|^{p}\ dx\\[2mm]
&+\int_{\{x_{N}\leq 1\}}|\xi(|x-a_{r}|)\eta(x_{N})-1|^{p}|\varphi(x-y)|^{p}\ dx
\end{array}
\]
and
\[
\begin{array}{ll}
\ \int_{B_{2\rho(a_{r})}}|\xi(|x-a_{r}|)\eta(x_{N})-1|^{p}|\varphi(x-y)|^{p}\ dx
&\leq C mesB_{2\rho(a_{r})}\max_{x\in\mathbb{R}^{N}}\varphi(x)\rightarrow 0\ as\ \rho \rightarrow 0,
\end{array}
\]
\[
\begin{array}{ll}
\ \int_{\{x_{N}\leq 1\}}|\xi(|x-a_{r}|)\eta(x_{N})-1|^{p}|\varphi(x-y)|^{p}\ dx
&=\int_{\{x_{N}\leq 1\}}|\eta(x_{N})-1|^{p}|\varphi(x-y)|^{p}\ dx\\[2mm]
&=\int_{\{z_{3}\leq y_{N}+1\}}|\eta (x_{N}+y_{N})-1|^{p}|\varphi(z)|^{p}\ dz\\[2mm]
& \rightarrow 0 \ as\ y_{N}\rightarrow +\infty,\  \rho \rightarrow 0.
\end{array}
\]
Therefore
$$
\|f_{y}-\varphi(x-y)\|_{L^{p}(\mathbb{R}^{N})}=o(1)\ y_{N} \rightarrow \infty\  and \ \rho \rightarrow 0.
$$

(ii)
Now, we claim that
$$
\int_{\mathbb{R}^{N}} \int_{\mathbb{R}^{N}} \frac{\left|(\xi(|x-a_{r}|)\eta(x_{N})-1) \varphi\left(x-y\right)-(\xi(|z-a_{r}|)\eta(z_{N})-1) \varphi\left(z-y\right)\right|^{2}}{|x-z|^{N+2 s}} d z d x=o_{n}(1)
$$
Indeed, let
$$
\Upsilon_{u}(x, y):=\frac{u(x)-u(z)}{|x-z|^{\frac{N}{2}+s}}
$$
Then, after the change of variables $\tilde{x}=x-y$ and $\tilde{y}=z-y,$ one has
\[\begin{array}{ll}
\int_{\mathbb{R}^{N}} \int_{\mathbb{R}^{N}} \frac{\left|(\xi(|x-a_{r}|)\eta(x_{N})-1) \varphi\left(x-y\right)-(\xi(|z-a_{r}|)\eta(z_{N})-1) \varphi\left(z-y\right)\right|^{2}}{|x-z|^{N+2 s}} d z d x
& =\int_{\mathbb{R}^{N}} \int_{\mathbb{R}^{N}}\left|\Upsilon_{n}(x, z)\right|^{2} d z d x.

\end{array}
\]
where
$$
\Phi_{n}(x, z):=\frac{\left(\xi(|x+y-a_{r}|)\eta(x_{N}+y_{N}))-1\right) \varphi(x)-\left(\xi(|z+y-a_{r}|)\eta(z_{N}+y_{N}))-1\right) \varphi(z)}{|x-z|^{\frac{N}{2}+s}}
$$
Recalling that  $|y-a_{r}|\rightarrow \infty$, and $y_{N}\rightarrow +\infty$, we also have
$$
\Upsilon_{n}(x, y) \rightarrow 0 \text { a.e. in } \mathbb{R}^{N} \times \mathbb{R}^{N}
$$
On the other hand, a direct application of the mean value theorem yields
\begin{eqnarray}\label{ezq:0b.1}
\begin{aligned}
\left|\Upsilon_{n}(x, z)\right| & \leq\left|\left(\xi(|x+y-a_{r}|)\eta(x_{N}+y_{N}))-1\right)  \| \Upsilon_{\varphi}(x, y)\right|+|\varphi(z)|\left|\Upsilon_{1-\xi\eta}\left(x+y, z+y\right)\right| \\
& \leq\left|\Upsilon_{\varphi}(x, z)\right|+\frac{C|\varphi(z)|}{|x-z|^{\frac{N}{2}+s-1}} \chi_{B(z, 1)}(x)+\frac{2|\varphi(z)|}{|x-z|^{\frac{N}{2}+s}} \chi_{B^{c}(z, 1)}(x),
\end{aligned}
\end{eqnarray}
for almost every $(x, z) \in \mathbb{R}^{N} \times \mathbb{R}^{N} .$ Now, it is easily seen that the right hand side in (\ref{ezq:0b.1}) is $L^{2}$ -integrable. Thus, By the Lebesgue's dominated convergence theorem and i, it follows that
$$\|f_{y}-\varphi(x-y)\|=o(1),|y-a_{r}|\rightarrow \infty,\  and\  y_{N}\rightarrow +\infty$$

By i,we have $\|f_{y}-\varphi(x-y)\|_{L^{2}(\mathbb{R}^{N})}=o(1)$, $|y-a_{r}|\rightarrow \infty$, and $y_{N}\rightarrow +\infty$, or $y_{N} \rightarrow \infty$ and $\rho \rightarrow 0$;
\[\begin{array}{ll}
\ \|f_{y}-\varphi(x-y)\|^{2}
&= \int_{\mathbb{R}^{N}} \int_{\mathbb{R}^{N}} \frac{\left|\left(f_{y}-\varphi(\cdot-y)\right)(x)-\left(f_{y}-\varphi(\cdot-y)\right)(z)\right|^{2}}{|x-z|^{N+2 s}} d z d x \\[2mm]
&+\int_{\mathbb{R}^{N}}\left|f_{y}(x)-\varphi(x-y)\right|^{2} d x .
\end{array}
\]
Setting
$$
I_{1}:=\iint_{\mathbb{R}^{2 N}} \frac{(\xi(|x-a_{r}|)\eta(x_{N}))-\xi(|z-a_{r}|)\eta(z_{3}))|^{2}|\varphi(x-y)|^{2}}{|x-z|^{N+2 s}} d z d x
$$
and
$$
I_{2}:=\iint_{\mathbb{R}^{2 N}} \frac{\left.\left|\xi(|z-a_{r}|)\eta(z_{3}))-1\right|^{2} \mid \varphi(x-y)-\varphi(z-y)\right)\left.\right|^{2}}{|x-z|^{N+2 s}} d z d x
$$
the following inequality holds
$$
 \int_{\mathbb{R}^{N}} \int_{\mathbb{R}^{N}} \frac{\left|\left(f_{y}-\varphi(\cdot-y)\right)(x)-\left(f_{y}-\varphi(\cdot-y)\right)(z)\right|^{2}}{|x-z|^{N+2 s}} d z d x \leq I_{1}+I_{2}.
$$
Moreover, by definition of $\xi,$ we also have
$$
\xi(|z+y-a_{r}|)\eta(z_{3}+y_{N})-\left.1\right|^{2} \frac{|\varphi(x)-\varphi(z)|^{2}}{|x-z|^{N+2 s}} \leq 4 \frac{|\varphi(x)-\varphi(z)|^{2}}{|x-z|^{N+2 s}} \in L^{1}\left(\mathbb{R}^{N} \times \mathbb{R}^{N}\right)
$$
and
$$\xi(|z+y-a_{r}|)\eta(z_{3}+y_{N})-1|^{2} \frac{|\varphi(x)-\varphi(z)|^{2}}{|x-y|^{N+2 s}} \rightarrow 0 \text { a.e. in } \mathbb{R}^{N} \times \mathbb{R}^{N}$$
as $y_{N} \rightarrow \infty$ and $\rho \rightarrow 0$.
Hence, the Lebesgue's theorem ensures that
$$
I_{2}\rightarrow 0 \text { as } \rho \rightarrow 0.
$$
Now, by [\cite{m21}, Lemma 2.3], for every $y \in \mathbb{R}^{N}$, one has
$$
I_{1}=\iint_{\mathbb{R}^{2 N}} \frac{(\xi(|x-a_{r}|)\eta(x_{N}))-\xi(|z-a_{r}|)\eta(z_{3}))|^{2}|\varphi(x-y)|^{2}}{|x-z|^{N+2 s}} d z d x \rightarrow 0 \text { as } \rho \rightarrow 0
$$.
Therefore
$$
\|f_{y}-\varphi(x-y)\|=o(1),  y_{N} \rightarrow +\infty , \rho \rightarrow 0.
$$

\end{proof}

\begin{lemma}
The equality $M_{\infty}=M$ holds true. Hence, there is no $u \in X_{0}^{s}$ such that $\|u\|^{2}=M$ and $\|u\|_{L^{p}\left(\mathbb{R}^{N}\right)}=1$, and so, the minimization problem (\ref{eq:1w.sq6}) does not have solution.
\end{lemma}

\begin{proof}
The proof is similar to \cite{4}, and we only give a sketch here.  By Proposition 2.3 - part (i) it follows that
$$
M_{\infty} \leq M
$$
 Take a sequence $y^{n}$ in $\Omega_{r}$ such that\\
\[|y^{n}-a_{r}|\rightarrow \infty,and\  y_{N}^{n} \rightarrow +\infty\   as\  n\rightarrow \infty.\]
Then by lemma \ref{lm:2.4}, we have
\[\|f_{y^{n}}-\varphi(x-y^{n})\|_{L^{p}(R^{N})}=o(1),|y^{n}-a_{r}|\rightarrow \infty,and\  y^{n}_{3}\rightarrow +\infty,\]
\[\|f_{y^{n}}-\varphi(x-y^{n})\|_{s}=o(1),|y^{n}-a_{r}|\rightarrow \infty, and \ y^{n}_{3}\rightarrow +\infty,\]
\[c_{y^{n}}=\frac{1}{\left\|f_{y^{n}}\right\|_{L^{p}\left(\mathbb{R}^{N}\right)}} \rightarrow 1,|y^{n}-a_{r}|\rightarrow \infty,and\  y^{n}_{3}\rightarrow +\infty\]
Now, since $\varphi$ is a minimizer of $(\ref{eq:1.s6}),$ one has
$$
\left\|f_{y^{n}}\right\|_{s}^{2}=\left\|\varphi\left(\cdot-y_{n}\right)\right\|_{s}^{2}+o_{n}(1)=\|\varphi\|_{s}^{2}+o_{n}(1)=M_{\infty}+o_{n}(1)
$$
Similar arguments ensure that
$$
\left\|\Psi_{n}\right\|_{s}^{2}=\left\|\Psi_{n}\right\|^{2}=M_{\infty}+o_{n}(1)
$$
and$$\left\|\Psi_{n}\right\|_{L^{p}\left(\mathbb{R}^{N}\right)}=1$$
So\[ M \leq M_{\infty}\].
We then conclude that $M= M_{\infty}$.
Now, suppose by contradiction that there is $v_{0} \in X_{0}^{s}$ satisfying
$$
\left\|v_{0}\right\|=M \text { and }\left\|v_{0}\right\|_{L^{p}(\Omega)}=1
$$
Without loss of generality, we can assume that $v_{0} \geq 0$ in $\Omega$. Note that by $M= M_{\infty}$, since $v_{0} \in$ $H^{s}\left(\mathbb{R}^{N}\right)$ and $\left\|v_{0}\right\|=\left\|v_{0}\right\|_{s},$ it follows that $v_{0}$ is a minimizer for  $(\ref{eq:1.s6}),$ and so, a solution of roblem
\begin{eqnarray}\label{eq:0.v1}
\left\{\begin{aligned}
(-\Delta)^{s} u+u &=M_{\infty} u^{p-1} \text { in } \mathbb{R}^{N} \\
u & \in H^{s}\left(\mathbb{R}^{N}\right) .
\end{aligned}\right.
\end{eqnarray}
Therefore, by the maximum principle we get that $v_{0}>0$ in $\mathbb{R}^{N}$, which is impossible, because $v_{0}=0$ in $\mathbb{R}^{N} \backslash \Omega_{r}$. This completes the proof.
\end{proof}

\section{A Compactness lemma }
In this section we prove a compactness result involving the energy functional $I: X_{0}^{s} \rightarrow \mathbb{R}$ associated to the main problem (\ref{eq:0.1}) and given by
$$
I(u):=\frac{1}{2}\left(\iint_{\mathcal{Q}} \frac{|u(x)-u(y)|^{2}}{|x-y|^{N+2 s}} d y d x+\int_{\Omega_{r}}|u|^{2} d x\right)-\frac{1}{p} \int_{\Omega_{r}}|u|^{p} d x
$$
In order to do this, we consider the problem
\begin{eqnarray}\label{eq:0.x1}
\left\{\begin{aligned}
(-\Delta)^{s} u+u &=|u|^{p-2} u \text { in } \mathbb{R}^{N} \\
u & \in H^{s}\left(\mathbb{R}^{N}\right)
\end{aligned}\right.
\end{eqnarray}
whose energy functional $I_{\infty}: H^{s}\left(\mathbb{R}^{N}\right) \rightarrow \mathbb{R}$ has the form
$$
I_{\infty}(u):=\frac{1}{2}\left(\int_{\mathbb{R}^{N}} \int_{\mathbb{R}^{N}} \frac{|u(x)-u(y)|^{2}}{|x-y|^{N+2 s}} d y d x+\int_{\mathbb{R}^{N}}|u|^{2} d x\right)-\frac{1}{p} \int_{\mathbb{R}^{N}}|u|^{p} d x.
$$
With the above notations we are able to prove the following compactness result.
\begin{lemma}\label{eq:0cb.1}
Let $\left\{u_{n}\right\} \subset X_{0}^{s}$ be a sequence such that
\begin{eqnarray}\label{eq:0bx.1}
I\left(u_{n}\right) \rightarrow c \text { and } I^{\prime}\left(u_{n}\right) \rightarrow 0 \text { as } n \rightarrow \infty.
\end{eqnarray}
 Then there are a nonnegative integer $k, k$ sequences $\left\{y_{n}^{i}\right\}$ of points of the form $\left(x_{n}^{\prime}, m_{n}+1 / 2\right)$ for integers $m_{n}, i=1,2, \cdots, k, $ $u_{0} \in X_{0}^{s}$ solving equation (\ref{eq:0.1}) and nontrivial functions $u^{1}, \cdots, u^{k}$ in $H^{s}\left(\mathbb{R}^{N}\right)$ solving equation (\ref{eq:0.x1}). Moreover there is a subsequence $\left\{u_{n}\right\}$ satisfying

$$
\begin{array}{l}
\text { (1) }u_{n}(x)=u^{0}(x)+u^{1}\left(x-x_{n}^{1}\right)+\cdots+u^{k}\left(x-x_{n}^{k}\right)+o(1) strongly, where
x_{n}^{i}=y_{n}^{1}+\cdots+y_{n}^{i} \rightarrow \infty,\\ i=1,2, \cdots, k \\
\text { (2) }\left\|u_{n}\right\|^{2}=\left\|u^{0}\right\|_{\Omega_{r}}^{2}+\left\|u^{1}\right\|^{2}+\cdots+\left\|u^{k}\right\|^{2}+o(1) \\
\text { (3) } I\left(u_{n}\right)=I\left(u^{0}\right)+I_{\infty}\left(u^{1}\right)+\cdots+I_{\infty}\left(u^{k}\right)+o(1)
\end{array}
$$
If $u_{n} \geqslant 0$ for $n=1,2, \cdots,$ then $u^{1}, \cdots, u^{k}$ can be chosen as positive solutions, and $u^{0} \geqslant 0$
\end{lemma}
 \begin{proof}
 See \cite{21,4}.
 \end{proof}

{\bf COROLLARY 2} Let $\left\{u_{n}\right\} \subset M_{\Omega_{r}}$ satisfy $\|u_{n}\|_{\Omega_{r}}^{2}=c+o(1)$ and
$M<c<2^{(p-2) / p} M.$ Then $\left\{u_{n}\right\}$ contains a strongly convergent subsequence.
\begin{proof}
 See \cite{21,4}.
 \end{proof}

\renewcommand{\theequation}{\thesection.\arabic{equation}}
\section{ Proof of Theorem 1}
Set
$$
\chi(t)=\left\{\begin{array}{ll}
1 & \text { if } 0 \leqslant t \leqslant 1 \\
\frac{1}{t} & \text { if } 1 \leqslant t<\infty
\end{array}\right.
$$
and define $\beta: H^{s}\left(\mathbb{R}^{N}\right) \rightarrow \mathbb{R}^{N}$\cite{4} by
$$
\beta(u)=\int_{\mathbb{R}^{N}} u^{2}(x) \chi(|x|) x d x.
$$
For $r \geqslant r_{1}$, let
$$
\begin{array}{l}V_{r}=\left\{\left.u \in H_{0}^{1}(\Omega_{r}\right)|\int_{\Omega_{r}}|u|^{p}=1, \beta(u)=a_{r}\right\} \\ c_{r}=\inf _{u \in V_{r}}\|u\|_{\Omega_{r}}^{2}.\end{array}
$$
\begin{lemma}
$c_{r}>M$.
\end{lemma}
ProoF: It is easy to see that $c_{r} \geqslant M.$ Suppose $c_{r}=\alpha .$ Take a sequence $\left\{v_{m}\right\} \subset X_{0}^{s}$ s.t.
$$
\begin{array}{l}
\left\|v_{m}\right\|_{L_{p} \left(\Omega_{r}\right)}=1, \beta\left(v_{m}\right)=a_{r} \quad \text { for } \quad m=1,2, \cdots, \\
\left\|v_{m}\right\|^{2}=M+o(1).
\end{array}
$$
Let $u_{m}=M^{1 /(p-2)} v_{m}$ for $m=1,2, \cdots$. Then
$$I^{\prime}\left(u_{n}\right)=o_{n}(1) \text { in }\left(X_{0}^{s}\right)^{*}$$
and
$$I\left(u_{n}\right)=\left(\frac{1}{2}-\frac{1}{p}\right) M^{\frac{p}{p-2}}+o_{n}(1).$$
By the maximum principle, $\left\{u_{m}\right\}$ does not contain any convergent subsequence. $\mathrm{By}$ lemma (\ref{eq:0cb.1}) there is a sequence $\left\{x_{m}\right\}$ of the form $\left(x_{m}^{\prime}, m+\frac{1}{2}\right)$ for integers $m$ such that
$$
\begin{array}{c}
\left|x_{m}\right| \longrightarrow \infty \\
u_{m}(x)=\varphi\left(x-x_{m}\right)+o(1) \text { strongly. }
\end{array}
$$
Since $\varphi$ is radially symmetric, we may take $m$ to be positive.
Next, we consider the following sets
$$
\mathbb{R}_{+}^{N}:=\left\{x \in \mathbb{R}^{N}:\left\langle x, x_{m}\right\rangle>0\right\} \text { and }\mathbb{R}_{-}^{N}:=\mathbb{R}^{N} \backslash\mathbb{R}_{+}^{N}.
$$
 We may assume that
$$
\begin{aligned}
\left|x_{m}\right| \geqslant 4 \text { from } m=1,2, \cdots . \text { Now } & \\
\qquad\left\langle\beta\left(\varphi\left(x-x_{m}\right)\right), x_{m}\right\rangle=& \int_{\mathbb{R}^{N}} \varphi^{2}\left(x-x_{m}\right) \chi(|x|)\left\langle x, x_{m}\right\rangle d x \\
=& \int_{\mathbb{R}_{+}^{N}} \varphi^{2}\left(x-x_{m}\right) \chi(|x|)\left\langle x, x_{m}\right\rangle d x \\
&+\int_{\left(\mathbb{R}_{-}^{N}\right)} \varphi^{2}\left(x-x_{m}\right) \chi(|x|)\left\langle x, x_{m}\right\rangle d x \\
\geqslant & \int_{B_{1}\left(x_{m}\right)} \varphi^{2}\left(x-x_{m}\right) \chi(|x|)\left\langle x, x_{m}\right\rangle d x \\
&+\int_{\mathbb{R}_{-}^{N}} \varphi^{2}\left(x-x_{m}\right) \chi(|x|)\left\langle x, x_{m}\right\rangle d x.
\end{aligned}
$$
Note that there are $c_{1}>0, c_{2}>0$ such that for $x \in B_{1}\left(x_{m}\right),$ we have
$$
\begin{array}{l}
\varphi^{2}\left(x-x_{m}\right) \geqslant c_{1} \\
\quad\left\langle x, x_{m}\right\rangle \geqslant c_{2}|x|\left|x_{m}\right| \text { for } m=1,2, \cdots .
\end{array}
$$
Thus
$$
\begin{aligned}
\int_{B_{1}\left(x_{m}\right)} \varphi^{2}\left(x-x_{m}\right) \chi(|x|)\left\langle x, x_{m}\right\rangle d x & \geqslant c_{1} c_{2} \int_{B_{1}\left(x_{m}\right)} \chi(|x|)|x|\left|x_{m}\right| d x \\
& \geqslant c_{3}\left|x_{m}\right|^{N+1}, \quad c_{3}>0 \quad \text { a constant. }
\end{aligned}
$$
Recalling that for each $x\in\mathbb{R}_{-}^{N}$,
$$
\left|x-y_{n}\right| \geq|x|
$$
it follows that
$$
\left|u\left(x-y_{n}\right)\right|^{2} \chi(|x|)|x| \leq R|u(|x|)|^{2} \in L^{1}\left(\mathbb{R}^{N}\right)(R>0)
$$
(see \cite{21} lemma 4.3).

This fact, combined with the limit
$$
u\left(x-y_{n}\right) \rightarrow 0 \text { as }\left|y_{n}\right| \rightarrow+\infty
$$
implies that
$$
\int_{\mathbb{R}_{-}^{N}}\left|u\left(x-y_{n}\right)\right|^{2} \chi(|x|)|x| d x=o_{n}(1).
$$
We conclude that
$$
\begin{aligned}
M^{1 /(p-2)}\left|a_{r}\right| & \geqslant\left\langle\beta\left(u_{m}\right), \frac{x_{m}}{\left|x_{m}\right|}\right\rangle \\
&=\left\langle\beta\left(\varphi\left(x-x_{m}\right)\right), \frac{x_{m}}{\left|x_{m}\right|}\right\rangle+o(1) \\
& \geqslant c_{3}\left|x_{m}\right|^{N}+o(1)
\end{aligned}
$$
a contradiction. Thus $c_{r}>M$.

 {\bf REMARK 1}  By Lemma \ref{lm:2.4}(1), there is $r_{1}>0$ such that
$$
\frac{1}{2} \leqslant\left\|f_{y}\right\|_{L^{p}\left(\Omega_{r}\right)} \leqslant \frac{3}{2}
$$
where $r \geqslant r_{1}$ and $\left|y-a_{r}\right| \geqslant r / 2$ and $y_{N} \geqslant r / 2$.

 {\bf REMARK 2}. By Lemma \ref{lm:2.4}(2), there is $r_{2} \geqslant r_{1}$ such that
$M<\left\|\Psi_{y}\right\|^{2}<\frac{c_{r}+M}{2}$
where $r \geqslant r_{2}$ and $\left|y-a_{r}\right| \geqslant r / 2$ and $y_{N} \geqslant r / 2$.
\begin{lemma}\label{eq:1.f3}
There is $r_{3} \geqslant r_{2}$ such that if $r \geqslant r_{3},$ then
$$
\left\langle\beta\left(\varphi_{y}\right), y\right\rangle>0 \quad \text { for } \quad y \in \partial\left(B_{r / 2}\left(a_{r}\right)\right).
$$
\end{lemma}

\begin{proof} By lemma \ref{lm:2.4}, we have $ 2 / 3 \leqslant c_{y} \leqslant 2 .$ For $r \geqslant r_{2},$ let
$$
\begin{aligned}
A_{\left((3 / 8) r_{1}(5 / 8) r\right)} &=\left\{x \in \mathbb{R}^{N}\left|\frac{3}{8} r \leqslant\right| x-a_{r} \mid \leqslant \frac{5}{8} r\right\}, \\
\mathbb{R}_{+}^{N}(y) &=\left\{x \in \mathbb{R}^{N} \mid\langle x, y\rangle>0\right\} \\
\mathbb{R}_{-}^{N}(y) &=\left\{x \in \mathbb{R}^{N} \mid\langle x, y\rangle<0\right\}
\end{aligned}
$$
$$\begin{aligned}
\left\langle\beta\left(\varphi_{y}\right), y\right\rangle=c_{y} &\left[\int_{\mathbb{R}_{+}^{N}(y)} \xi^{2}\left(\left|x-a_{r}\right|\right) \eta^{2}\left(x_{N}\right) \varphi^{2}(x-y) \chi(|x|)\langle x, y\rangle d x\right.\\
&\left.+\int_{\mathbb{R}_{-}^{N}(y)} \xi^{2}\left(\left|x-a_{r}\right|\right) \eta^{2}\left(x_{N}\right) \varphi^{2}(x-y) \chi(|x|)\langle x, y\rangle d x\right] \\
\geqslant \frac{2}{3} &\left[\int_{A((3 / 8) r,(5 / 8) r)} \varphi^{2}(x-y) \chi(|x|)\langle x, y\rangle\right.\\
&+\int_{\mathbb{R}_{-}^{N}(y)} \varphi^{2}(x-y) \chi(|x|)\langle x, y\rangle d x
\end{aligned}$$
$$\begin{aligned}
\int_{A((3 / 8) r,(5 / 8) r)} \varphi^{2}(x-y) \chi(|x|)\langle x, y\rangle d x & \geqslant c_{6} \int_{A((3 / 8) r,(5 / 8) r)} \chi(|x|)|x||y| d x \text { for } c_{6}>0 \\
& \geqslant c_{6}|y|\left[\left(\frac{5}{8} r\right)^{N}-\left(\frac{3}{8} r\right)^{N}\right] \\
& \geqslant c_{7} r^{N+1} \text { for } c_{7}>0 . \\
\int_{R_{-}^{N}(y)} \varphi^{2}(x-y) \chi(|x|)\langle x, y\rangle d x =o_{n}(1).
\end{aligned}$$
Therefore, there is $r_{3} \geqslant r_{2},$ such that if $r \geqslant r_{3},\left|y-a_{r}\right|=r / 2$
$$
\left\langle\beta\left(\Psi_{y}\right), y\right\rangle \geqslant c_{7} r^{N+1}-o_{n}(1)>0.
$$
This completes the proof.
 \end{proof}
By Lemma \ref{lm:2.4} and Lemma \ref{eq:1.f3}, fix $\rho_{0}>0, r_{0} \geqslant r_{3}$ such that if $0<\rho \leqslant \rho_{0}, r \geqslant r_{0}$ then $\left\|\varphi_{y}\right\|_{\Omega_{r}}^{2}<2^{(p-2) / p} \alpha$ for $y \in \overline{B_{r / 2}}\left(a_{r}\right) .$ From now on, fix $\rho_{0}, r_{0},$ for $r \geqslant r_{0} .$ Let
$$
\begin{array}{l}
B=\left\{\Psi_{y}|| y-a_{r} \mid \leqslant \frac{r}{2}\right\} \\
\Gamma=\left\{h \in C\left(V_{r}, V_{r}\right) \mid h(u)=u \quad \text { if } \quad\|u\|^{2}<\frac{c_{r}+\alpha}{2}\right\}.
\end{array}
$$
\begin{lemma}
$$h(B) \cap V_{r} \neq \emptyset \text { for each } h \in \Gamma.$$
\end{lemma}
 \begin{proof}  Let $h \in \Gamma$ and $H(x)=\beta \circ h \circ \varphi_{x}: \mathbb{R}^{N} \rightarrow \mathbb{R}^{N}$. Consider the homotopy,
for $0 \leqslant t \leqslant 1$
$$
F(t, x)=(1-t) H(x)+t I(x) \quad \text { for } \quad x \in \mathbb{R}^{N}.
$$
If $x \in \partial\left(B_{r / 2}\left(a_{r}\right)\right),$ then, by Remark 8 and Lemma 9,
$$
\begin{array}{c}
\left\langle\beta\left(\Psi_{x}\right), x\right\rangle>0 \\
\alpha<\left\|\Psi_{x}\right\|^{2}<\frac{c_{r}+\alpha}{2}.
\end{array}
$$
Then
$$
\begin{aligned}
\langle F(t, x), x\rangle &=\langle(1-t) H(x), x\rangle+\langle t x, x\rangle \\
&=(1-t)\left\langle\beta\left(\Psi_{x}\right), x\right\rangle+t\langle x, x\rangle \\
&>0.
\end{aligned}
$$
Thus $F(t, x) \neq 0$ for $x \in \partial\left(B_{r / 2}\left(a_{r}\right)\right) .$ By the homotopic invariance of the degree
$$
d\left(H(x), B_{r / 2}\left(a_{r}\right), a_{r}\right)=d\left(I, B_{r / 2}\left(a_{r}\right), a_{r}\right)=1.
$$
There is $x \in B_{r / 2}\left(a_{r}\right)$ such that
$$
a_{r}=H(x)=\beta\left(h \circ \Psi_{x}\right).
$$
Thus $h(B) \cap V_{r} \neq \emptyset$ for each $h \in \Gamma$.
Now we are in the position to prove Theorem A: Consider the class of mappings
$$
F=\left\{h \in C\left(\overline{B_{r / 2}\left(a_{r}\right)}\right), H^{1}\left(R_{N}\right):\left.h\right|_{\partial B_{r / 2}\left(a_{r}\right)}=\Psi_{y}\right\}
$$
and set
$$
c=\inf _{h \in F} \frac{\sup }{y \in B_{r / 2}\left(a_{r}\right)}\|h(y)\|_{\Omega_{r}}^{2}.
$$
It follows from the above Lemmas, with the appropriate choice of $r$ that
$$
\alpha<c_{r}=\inf _{u \in V_{\gamma}}\|u\|_{\Omega_{r}}^{2} \leqslant c<2^{(p-2) / p} \alpha
$$
and
$$
\max _{\partial B_{r / 2}\left(a_{r}\right)}\|h(y)\|_{\Omega_{r}}^{2}<\max _{B_{r / 2}\left(a_{r}\right)}\|h(y)\|_{\Omega_{r}}^{2}.
$$
Theorem 1,1 then follows by applying the version of the mountain pass theorem from Brezis-Nirenberg \cite{d6}.
 \end{proof}

\end{document}